\documentclass{article}

\usepackage{amsmath,amsthm,amssymb,amsfonts}
\usepackage{verbatim}

\usepackage{stmaryrd} 
\usepackage{hyperref}

\usepackage[colorinlistoftodos]{todonotes}

\usepackage{tikz}
\usepackage{tkz-berge, tkz-graph}
\tikzset{vertex/.style={circle,draw,fill,inner sep=0pt,minimum size=1mm}}

\theoremstyle{plain}
\newtheorem{thm}{Theorem}
\newtheorem{lem}[thm]{Lemma}
\newtheorem{prop}[thm]{Proposition}
\newtheorem{cor}[thm]{Corollary}

\newtheorem{remark}[thm]{Remark}

\theoremstyle{definition}
\newtheorem{definition}[thm]{Definition}
\newtheorem{exl}[thm]{Example}

\numberwithin{thm}{section}

\newcommand{\adj}{\leftrightarrow}

\def\Z{{\mathbb Z}}
\def\N{{\mathbb N}}
\def\R{{\mathbb R}}

\begin{document}
\title{Remarks on Fixed Point Assertions in Digital Topology, 2}
\author{Laurence Boxer
         \thanks{
    Department of Computer and Information Sciences,
    Niagara University,
    Niagara University, NY 14109, USA;
    and Department of Computer Science and Engineering,
    State University of New York at Buffalo.
    E-mail: boxer@niagara.edu
    }
    }
\date{ }
\maketitle

\begin{abstract}
Several recent papers in digital topology have sought to
obtain fixed point results by mimicking the use of tools
from classical topology, such as complete metric spaces. We show that
in many cases, researchers using these tools have derived
conclusions that are incorrect, trivial, or limited.

\end{abstract}

\section{Introduction}
This paper continues the work of~\cite{BxSt18} and quotes or paraphrases from it.

Recent papers have attempted to apply to digital images
ideas from Euclidean topology and real analysis concerning 
metrics and fixed points. While the underlying motivation
of digital topology comes from Euclidean topology and real
analysis, some applications of fixed point theory recently featured in the 
literature of digital topology seem of doubtful worth. Although papers
including~\cite{Rosenfeld87,Bx-etal} have valid and
interesting results for fixed points and for ``almost" or
``approximate" fixed points in digital topology,
many other published assertions concerning fixed points in 
digital topology are incorrect, trivial (e.g., applicable
only to singletons, or only to constant functions), or limited,
as discussed in~\cite{BxSt18}. After submitting~\cite{BxSt18},
we learned of several additional publications with assertions 
characterized as above; these are discussed in the current paper.

The less-than-ideal papers we discuss have in common a definition of a 
digital metric space $(X,d,\kappa)$, where $X$ is a set of lattice points, 
$d$ is a metric (typically, the Euclidean), and $\kappa$ is an 
adjacency relation on $X$, making
$(X,\kappa)$ a graph; and then these papers never make use of $\kappa$.
Functions considered usually are all continuous in the topological sense, 
since the metric $d$ usually imposes a discrete topology on the digital 
image; but are often discontinuous in the digital sense of preserving
graph connectedness. Many of these papers' assertions are modifications
of results known for the Euclidean topology of $\R^n$ that tell us 
little or nothing about digital images as graphs. Some of these papers' 
assertions are of interest if we regard the functions investigated as 
defined on subsets of $\R^n$. In many cases, we offer corrections, 
notes on their limitations, or improvements.

\section{Preliminaries}
We let $\Z$ denote the set of integers, and $\R$, the real line.

We consider a digital image as a graph $(X,\kappa)$, where 
$X \subset \Z^n$ for some positive integer $n$ and
$\kappa$ is an adjacency relation on $X$. We will often assume
that $X$ is a finite set, as in the ``real world."

A digital metric space is~\cite{EgeKaraca} a triple
$(X,d,\kappa)$ where $(X,\kappa)$ is a digital image and
$d$ is a metric for $X$. In~\cite{EgeKaraca}, $d$ was taken
to be the Euclidean metric, as was the case in many subsequent papers,
but we will not limit our
discussion to the Euclidean metric. Often, however, we will assume
$d$ is an $\ell_p$ metric (see section~\ref{l-p-def}).

The {\em diameter} of a metric space $(X,d)$
is 
\[diam \, X = \max \{d(x,y) \, | \, x,y \in X\}.\]

\subsection{Adjacencies}
The most commonly used adjacencies for digital images
are the $c_u$-adjacencies, defined as follows.
\begin{definition}
Let $p,q \in \Z^n$, $p=(p_1,\ldots,p_n)$, 
$q=(q_1,\ldots,q_n)$, $p \ne q$. Let $1 \le u \le n$. We say
$p$ and $q$ are $c_u$-adjacent, denoted
$p \adj_{c_u} q$ or $p \adj q$ when the adjacency is
understood, if
\begin{itemize}
    \item for at most $u$ distinct indices $i$,
          $|p_i - q_i| = 1$, and
    \item for all other indices $j$, $p_j = q_j$.
\end{itemize}
\end{definition}

Often, a $c_u$-adjacency is denoted by the number
of points in $\Z^n$ that are $c_u$-adjacent to a given
point. E.g.,
\begin{itemize}
    \item in $\Z^1$, $c_1$-adjacency is 2-adjacency;
    \item in $\Z^2$, $c_1$-adjacency is 4-adjacency and
          $c_2$-adjacency is 8-adjacency;
    \item in $\Z^3$, $c_1$-adjacency is 8-adjacency,
          $c_2$-adjacency is 18-adjacency, and
          $c_3$-adjacency is 26-adjacency.
\end{itemize}

Other adjacencies for digital images are discussed in
papers such as \cite{Herman,Bx17,Bx18}.

A {\em digital interval} is a digital image of the form
$([a,b]_Z, 2)$, where $a < b$ and 
$[a,b]_Z = \{z \in \Z \, | \, a \le z \le b\}$.

\subsection{$\ell_p$ metric}
\label{l-p-def}
Let $X \subset \R^n$ and let $x=(x_1,\ldots, x_n)$ and
$y=(y_1,\ldots,y_n)$ be points of $X$. Let
$1 \le p \le \infty$. The $\ell_p$ metric $d$ for $X$ is defined by
\[ d(x,y) = \left \{ \begin{array}{ll}
        \left (    \sum_{i=1}^n |x_i - y_i|^p \right )^{1/p} &
            \mbox{for } 1 \le p < \infty; \\
            \max \{|x_i - y_i|\}_{i=1}^n &
            \mbox{for } p = \infty. \end{array}
\right .
\]
For $p=1$, this gives us the {\em Manhattan metric}
$d(x,y) = \sum_{i=1}^n |x_i - y_i|$; for $p=2$, we have the
{\em Euclidean metric} $d(x,y) = (\sum_{i=1}^n |x_i - y_i|^2)^{1/2}$.

The following are easily proved.

\begin{prop}
Let $x,y \in \Z^n$ and let $d$ be any $\ell_p$ metric. Then
\begin{itemize}
\item if $d(x,y) < 1$, then $x=y$;
\item if $1 \le u \le n$ and $x \adj_{c_u} y$, then
      $d(x,y) \le u^{1/p}$.
\end{itemize}
\end{prop}

\subsection{Digital continuity}
\begin{definition}
\label{digitalContinuity}
\cite{Rosenfeld87,Bx99}
A function $f: (X,\kappa) \to (Y, \lambda)$ between
digital images is $(\kappa,\lambda)$-{\em continuous}
(or just {\em continuous} when $\kappa$ and $\lambda$ are
understood) if for every $\kappa$-connected subset 
$X'$ of $X$, $f(X')$ is a $\lambda$-connected subset of $Y$.
\end{definition}

\begin{thm}
{\rm \cite{Bx99}}
A function $f: (X,\kappa) \to (Y, \lambda)$ between
digital images is $(\kappa,\lambda)$-{\em continuous} if
and only if $x \adj_{\kappa} x'$ in $X$ implies either
$f(x)=f(x')$ or $f(x) \adj_{\lambda} f(x')$ in $Y$.
\end{thm}

\subsection{Cauchy sequences and complete metric spaces}
The papers~\cite{DalalEtAl,EgeKaraca,Han16,Hossain-etal,Jain,JR17b,JR17,JR18,Mishra-etal,Ege-etal,RaniJyo,SrideviKKb,SrideviKK} apply to 
digital images the notions of Cauchy sequence and complete metric space.
Since if the digital image $X$ is finite or uses a common metric $d$ such as an $\ell_p$ 
metric, the digital metric space $(X,d,\kappa)$ is a discrete topological space,
the digital versions of these notions are quite limited.

Recall that a sequence of points $\{x_n\}$ in a metric space 
$(X,d)$ is a {\em Cauchy sequence}
if for all $\varepsilon > 0$ there exists $n_0 \in \N$
such that $m,n > n_0$ implies
$d(x_m,x_n) < \varepsilon$. If every Cauchy sequence in $X$ has a limit, then $(X,d)$ is a 
{\em complete metric space}.

It has been shown that under a mild additional assumption, a digital
Cauchy sequence is eventually constant.

\begin{thm}
{\rm ~\cite{Han16,BxSt18}}
\label{Han-Cauchy}
Let $a > 0$. If $d$ is a metric on a digital image $(X,\kappa)$ such that
for all distinct $x,y \in X$ we have $d(x,y) > a$, then
for any Cauchy sequence $\{x_i\}_{i=1}^{\infty} \subset X$ 
there exists $n_0 \in \N$ such that $m,n > n_0$ implies $x_m = x_n$.
\end{thm}

An immediate consequence of Theorem~\ref{Han-Cauchy} is the 
following.

\begin{cor}
{\rm \cite{Han16}}
\label{Han-Cauchy-cor}
Let $(X,d,\kappa)$ be a digital metric space. If $d$ is a metric on $(X,\kappa)$ such 
that for all distinct $x,y \in X$ we have $d(x,y) > a$ for some constant $a > 0$, 
then any Cauchy sequence in $X$ is eventually constant, and $(X,d)$ is a complete metric space.
\end{cor}

\begin{remark}
{\rm \cite{BxSt18}
It is easily seen that the hypotheses of 
Theorem~\ref{Han-Cauchy} and 
Corollary~\ref{Han-Cauchy-cor} 
are satisfied for any finite digital metric space, or for a 
digital metric space $(X,d,\kappa)$ for which the metric $d$ 
is any $\ell_p$ metric. Thus, a Cauchy sequence that is
not eventually constant can only occur in an infinite
digital metric space with an unusual metric. Such an
example is given below.
}
\end{remark}

\begin{exl}
{\rm \cite{BxSt18}}
Let $d$ be the metric on $(\N,c_1)$ defined by
$d(i,j) = |1/i - 1/j|$. Then $\{i\}_{i=1}^{\infty}$ is
a Cauchy sequence for this metric that does not have a
limit.
\end{exl}

\subsection{Function sets $\Psi$, $\Phi$}
Below, we define sets of functions $\Psi, \Phi$ that will
be used in the following.

\begin{definition}
{\rm \cite{JR17b}}
\label{psi1Family}
Let $\Psi$ be the set of functions 
$\psi: [0,\infty) \to [0,\infty)$ such that for each
$\psi \in \Psi$ we have
\begin{itemize}
    \item $\psi$ is nondecreasing, and
    \item $\sum_{n=1}^{\infty} \psi^n(t) < \infty$
          for all $t > 0$, where $\psi^n$ represents the
          $n$-fold composition of $\psi$.
\end{itemize}
\end{definition}

\begin{definition}
{\rm \cite{SrideviKKb}}
Let $\Phi$ be the set of functions $\phi: [0,\infty) \to [0,\infty)$
such that $\phi$ is increasing, $\phi(t) = 0$ if and only if $t=0$,
and $\phi(t) < t$ for $t > 0$.
\end{definition}

\begin{prop}
\label{psi1UpperBound}
Let $\psi \in \Psi$. Then for all $t > 0$, $\psi(t) < t$.
\end{prop}

\begin{proof}
Suppose there exists $t_0 > 0$ such that $\psi(t_0) \ge t_0$. Since 
$\psi$ is nondecreasing, an easy induction yields that
$\psi^{n+1}(t_0) \ge \psi^n(t_0)$ for all $n \in \N$. Therefore,
$\sum_{n=1}^{\infty} \psi^n(t_0) = \infty$, contrary to Definition~\ref{psi1Family}.
The contradiction establishes the assertion.
\end{proof}

A notion often used with the set~$\Psi$ is given by the following.

\begin{definition}
{\rm \cite{SametEtAl}}
\label{alphaAdmissible}
Let $T: X \to X$ and $\alpha: X \times X \to [0, \infty)$.
We say $T$ is {\em $\alpha$-admissible} if
$\alpha(x,y) \ge 1$ implies $\alpha(T(x),T(y)) \ge 1$.
\end{definition}

\subsection{An example}
\begin{exl}
\label{counter}
{\rm \cite{BxSt18}}
Let 
\[ X = \{p_1 = (0,0,0,0,0), ~~ p_2 = (2,0,0,0,0), ~~ p_3 = (1,1,1,1,1)\} \subset \Z^5. \]
Let $f: X \to X$ be defined by $f(p_1)=f(p_2) = p_1$, $f(p_3)=p_2$.
Then $f$ is not $(c_5,c_5)$-continuous. 
\end{exl}

Despite its discontinuity, the function of Example~\ref{counter} was shown
in~\cite{BxSt18} to exemplify many different types of functions studied
in~\cite{EgeKaraca,Han16,Hossain-etal,Jain,JR17,JR18,Mishra-etal,Ege-etal,RaniJyo}. In the
current paper, this example is also used to show that functions of the type studied need
not be digitally continuous.

\section{Various compatibilities in \cite{DalalEtAl}}
\subsection{Equivalence of compatibilities}
In this section, we examine the equivalence of several different types of ``compatible''
functions discussed in~\cite{DalalEtAl}.

\begin{definition}
\label{DalalMaps}
{\rm ~\cite{DalalEtAl}}
Suppose $S$ and $T$ are self-maps on a digital metric space $(X,d,\kappa)$. Suppose
$\{x_n\}_{n=1}^{\infty}$ is a sequence in $X$ such that
\begin{equation}
\label{seqEqn}
lim_{n \to \infty} S(x_n) = lim_{n \to \infty} T(x_n) = t \mbox{ for some } t \in X.
\end{equation}
We have the following.
\begin{itemize}
\item $S$ and $T$ are called {\em compatible} if 
      $lim_{n \to \infty} d(S(T(x_n)), T(S(x_n))) = 0$ for all
      sequences $\{x_n\}_{n=1}^{\infty} \subset X$ that satisfy
      statement~(\ref{seqEqn}).
\item $S$ and $T$ are called {\em compatible of type (A)} if 
      $lim_{n \to \infty} d(S(T(x_n)), T(T(x_n))) = 0 = lim_{n \to \infty} d(T(S(x_n)), S(S(x_n)))$ for all
      sequences $\{x_n\}_{n=1}^{\infty} \subset X$ that satisfy
      statement~(\ref{seqEqn}).
\item $S$ and $T$ are called {\em compatible of type (P)} if
      $lim_{n \to \infty} d(S(S(x_n)), T(T(x_n))) = 0$ for all
      sequences $\{x_n\}_{n=1}^{\infty} \subset X$ that satisfy
      statement~(\ref{seqEqn}).
\end{itemize}
\end{definition}

Because digital metric spaces are typically discrete, these properties can 
be simplified as shown in the remainder of this section.

\begin{prop}
\label{DalCompatibleAareCompatible}
{\rm \cite{DalalEtAl}} Let $S$ and $T$ be compatible maps of type (A) on a digital metric
space $(X,d,\kappa)$. If one of $S$ and $T$ is continuous, then $S$ and $T$ are compatible.
\end{prop}

The proof given for Proposition~\ref{DalCompatibleAareCompatible} 
clarifies that the continuity expected of $S$ or $T$ is in the sense of
the classical ``$\varepsilon - \delta$
definition", not in the sense of digital continuity. However, it turns out
that this assumption is usually unnecessary, as shown below.

\begin{thm}
\label{compatEquivs}
Let $(X,d,\kappa)$ be a digital metric space, where either $X$ is finite 
or $d$ is an $\ell_p$ metric. Let $S$ and $T$ be self-maps on $X$. 
Then the following are equivalent.
\begin{itemize}
\item $S$ and $T$ are compatible.
\item $S$ and $T$ are compatible of type (A).
\item $S$ and $T$ are compatible of type (P).
\end{itemize}
\end{thm}

\begin{proof}
Throughout this proof, let $\{x_n\}_{n=1}^{\infty} \subset X$ satisfy
statement~(\ref{seqEqn}).

Suppose $S$ and $T$ are compatible. Then $d(S(S(x_n)), T(T(x_n))) \le$
\[  d(S(S(x_n)),S(T(x_n))) + 
   d(S(T(x_n)),T(S(x_n))) + d(T(S(x_n)),T(T(x_n))). 
\]
By Corollary~\ref{Han-Cauchy-cor} and compatibility, the right side
of the latter inequality
\[ \to_{n \to \infty} d(S(t), S(t)) + 0 + d(T(t), T(t)) = 0.
\]
Thus, $S$ and $T$ are compatible of type (P).

Suppose $S$ and $T$ are compatible of type (P). Then, using Corollary~\ref{Han-Cauchy-cor},
\[ lim_{n=1}^{\infty} d(S(T(x_n)), T(T(x_n))) = lim_{n=1}^{\infty} d(S(t), T(T(x_n))) =
\]
\[ lim_{n=1}^{\infty} d(S(S(x_n)), T(T(x_n))) = \mbox{ (because compatible of type (P)) } 0.
\]
Similarly, $lim_{n=1}^{\infty} d(T(S(x_n)), S(S(x_n))) = 0$. Therefore, 
$S$ and $T$ are compatible of type (A).

If $S$ and $T$ are compatible of type (A), then,
using the triangle inequality and Definition~\ref{DalalMaps},
\[ d(S(T(x_n)), T(S(x_n))) \le d(S(T(x_n)), T(T(x_n))) + d(T(T(x_n))), T(S(x_n)))
\]
\[ \to_{n \to \infty} 0 + 0 = 0.
\]
Hence, $S$ and $T$ are compatible.
\end{proof}

\subsection{Compatible functions' common fixed points}
The assertions stated as Theorem~23 and Theorem~24 of~\cite{DalalEtAl}
are concerned with the existence
of a common fixed point of four self-maps of a digital metric space. However, these
assertions are incorrect. We give a counterexample below.

Stated as Theorem~23 of~\cite{DalalEtAl} is the following.
\begin{quote}
{\em Let $A,B,S$, and $T$ be self-maps on a complete digital metric space
$(X,d,\kappa)$ such that

(a) $S(X) \subset B(X)$ and $T(X) \subset A(X)$;

(b) the pairs $(A,S)$ and $(B,T)$ are compatible;

(c) one of $S,T,A$, and $B$ is continuous; and

(d) we have 
\[ F[d(A(x),B(y)), d(S(x),T(y)), d(A(x),S(x)),d(B(y),T(y)),\] 
\[ d(A(x),T(y)), d(B(y),S(x))] \le 0\]
   for all $x,y \in X$ {\rm [``$\le 0$'' was omitted in the statement of 
   the assertion, but the argument given as proof clarifies that this was 
   intended]}.

Then $A,B,S$, and $T$ have a unique common fixed point.
} 
\end{quote}

Stated as Theorem~24 of~\cite{DalalEtAl} is the following.
\begin{quote}
{\em Let $A,B,S$, and $T$ be self-maps on a complete digital metric space
$(X,d,\kappa)$ satisfying conditions (a), (c), and (d). If the pairs $(A,S)$ and
$(B,T)$ are compatible of type (A) or of type (P) 
then $A,B,S$, and $T$ have a unique common fixed point.
} 
\end{quote}
By Theorem~\ref{compatEquivs}, if $d$ is an $\ell_p$ metric then these
assertions are equivalent.
A counterexample to these assertions:

\begin{exl}
Let $X = \Z$, $d(x,y)= |x-y|$, $\kappa = c_1$, $A(x)=B(x)= x-1$, $S(x)=T(x)=x+1$,
$F(x_1,x_2,x_3,x_4,x_5,x_6) = 0$.
\end{exl}

\begin{proof}
That this is a counterexample to the assertion stated as Theorem~23 
of~\cite{DalalEtAl} is shown as follows. Clearly 
$S(X)=B(X)= \Z = T(X) = A(X)$. The pair $(A,S)$ is
compatible, since $A(S(x))=x=S(A(x))$ for 
all $x \in X$; similarly, $(B,T)$ is a compatible pair. All of $S,T,A$, 
and $B$ are continuous in both the ``$\varepsilon - \delta$'' sense 
and in the digital sense with respect to the $c_1$ adjacency. Trivially,
\[ F(d(Ax,By), d(Sx,Ty), d(Ax,Sx),d(By,Ty), d(Ax,Ty), d(By,Sx)) \le 0, \]
for all $x,y \in X$. However, none of $S,T,A$, and $B$ has a fixed point.
\end{proof}

\section{Expansive mappings in \cite{JR17b}}
The paper~\cite{JR17b} is concerned with fixed points for expansive maps on
digital metric spaces.

\begin{definition}
{\rm \cite{JR17b}}
\label{expansive}
Let $T$ be a self map on a complete metric space $(X,d)$ such that
$T$ is onto, and for some $k \ge 1$ and all $x,y \in X$, 
\begin{equation}
\label{expansiveIneq}
d(T(x),T(y)) \ge k \, d(x,y).
\end{equation}
Then $T$ is called an {\em expansive map}.
\end{definition}

\begin{remark}
In~{\rm \cite{JR18}}, the constant $k$ of~{\rm (\ref{expansiveIneq})}
was restricted to $k > 1$. Theorem~\ref{expansionPreservesDist} below
shows there is no such map if $X$ is finite.
\end{remark}

The following shows a limitation on the application of expansive maps in digital topology.
The result seems contrary to the spirit of Definition~\ref{expansive}.

\begin{thm}
\label{expansionPreservesDist}
If $T$ is an expansive map on a finite digital image $(X,d,\kappa)$, then for
all $x,y \in X$, $d(T(x),T(y)) = d(x,y)$.
\end{thm}

\begin{proof}
It is shown at Theorem~4.9 of~\cite{BxSt18} that $T$ cannot 
satisfy~(\ref{expansiveIneq}) for $k>1$. Thus, we have $k=1$, so 
\begin{equation}
\label{constIs1}
d(T(x),T(y)) \ge d(x,y) \mbox{ for all } x,y \in X.
\end{equation}

Since $X$ is finite, there exists a maximal finite set 
$\{d_i\}_{i=1}^m \in (0,\infty)$
such that $0 < d_1 < d_2 < \ldots < d_m$ and sets
\[ S_i = \{\{u,v\} \in X^2 \, | \, d(u,v) = d_i\} \neq \emptyset.
\]
Suppose that there exist $x,y \in X$ such that $d(T(x),T(y)) > d(x,y)$. Then
there exists $j$ such that
\[ j = \min\{i \in \{1,\ldots,m\} \, | \, d(T(x_0),T(y_0)) > d(x_0,y_0) \mbox{ for some }
             \{x_0,y_0\} \in S_i\}.
\]

Thus $\{ \{T(x),T(y)\} \, | \, \{x,y\} \in S_j \} \not \subset S_j$.
But $T$ is onto and $X$ is finite, so there exist
$x_1,y_1 \in X$ such that $\{x_1,y_1\} \not \in S_j$ and
$\{T(x_1),T(y_1)\} \in S_j$.
By our choice of $j$, there exists an index $k>j$ such that
$\{x_1,y_1\} \in S_k$. Therefore,
\[d(T(x_1), T(y_1)) = d_j < d_k = d(x_1,y_1), \]
a contradiction of~(\ref{constIs1}). This establishes the assertion.
\end{proof}

Even being an isomorphism need not make a self-map expansive, as shown
in the following.

\begin{exl}
\label{expansiveExl}
Let $X = \{p_0=(0,0), p_1=(1,0), p_2=(1,1)\} \subset \Z^2$. Let
$f: (X,c_2) \to (X,c_2)$ be the rotation defined by
\[ f(p_i) = p_{(i+1) \mod 3}.\]
Then it is easily seen that $f$ is a $(c_2,c_2)$-isomorphism. However, if we let $d$
be any $\ell_p$ metric, then by Theorem~\ref{expansionPreservesDist}, $f$ is not
an expansive mapping, since 
\[ d(p_1,p_2) = 1 \neq 2^{1/p} = d(p_2, p_0) = d(f(p_1),f(p_2)). \]
\end{exl}

\begin{definition}
\label{genAlphaPsi}
{\rm \cite{JR17b}}
Let $(X,d,\kappa)$ be a digital metric space and let $T: X \to X$ be a mapping. 
We say that $T$ is a {\em generalised $\alpha$-$\psi$-expansive mapping} if there exist
functions $\alpha: X \times X \to [0,\infty)$ and $\psi \in \Psi$ such that 
for all $x,y \in X$ we have
\begin{equation}
\label{generalizedExpansive}
\psi(d(T(x),T(y))) \ge \alpha(x,y) M(x,y) 
\end{equation}
where
\[ M(x,y) = \max \{d(x,y), \frac{d(x,T(x)) + d(y, T(y))}{2},
              \frac{d(x,T(y)) + d(y, T(x))}{2} \}.
\]
\end{definition}

\begin{remark}
Any function $T: X \to X$ on a digital metric space is a generalized
$\alpha$-$\psi$-expansive mapping if $\alpha$ is the constant function 
with value $0$. Therefore, a generalised $\alpha$-$\psi$-expansive 
mapping need not be digitally continuous.
\end{remark}

If one desires an example of a discontinuous generalised 
$\alpha$-$\psi$-expansive mapping for which $\alpha$ is not the constant 
function with value $0$, consider the following. This example 
also shows that the status of a map as an expansive
map depends on the metric used, and that an expansive map need not be
digitally continuous.

\begin{exl}
\label{expansiveNonCont}
Let $X = \{p_0, p_1, p_2\} \subset \Z^2$, where
$p_0 = (0,0)$, $p_1 = (1,1)$, $p_2 = (2,0)$. Let
$T: X \to X$ be the circular rotation $T(p_i) = p_{(i+1) \mod 3}$.
Then 
\begin{itemize}
\item $T$ is an expansive map with respect to the Manhattan metric, but
not with respect to the Euclidean metric;
\item $T$ is not $(c_2,c_2)$-continuous;
\item $T$ is a generalised $\alpha$-$\psi$-expansive mapping for
$\psi(t)=t/2$ and $\alpha(x,y) = 1/3$.
\end{itemize}
\end{exl}

\begin{proof}
With respect to the Manhattan metric,
\begin{equation}
\label{dist2}
i \neq j \mbox{ implies } d(p_i,p_j)=2.
\end{equation}
Since $T$ is a bijection, we have
$d(T(p_i),T(p_j)) = d(p_i,p_j)$, so $T$ is an expansive map.

With respect to the Euclidean metric, 
\[d(p_0,p_2) = 2 > \sqrt{2} = d(p_1,p_0) = d(T(p_0), T(p_2)),
\]
so $T$ is not an expansive map.

Since
$p_1 \adj_{c_2} p_2$ but $T(p_1) = p_2$ and $T(p_2) = p_0$ are neither
equal nor $c_2$-adjacent, $T$ is not $c_2$-continuous.

Using the Manhattan metric, we have from~(\ref{dist2}) 
that $i \neq j$ implies
\[ d(f(x_i),f(x_j)) = 2 =d(x_i, x_j).\]
Therefore, from Definition~\ref{genAlphaPsi}  we have
$M(x_i,x_i)=0$ and $i \neq j$ implies $M(x_i,x_j) = 2$.
Then one sees easily that $T$ is a generalised $\alpha$-$\psi$-expansive
mapping for $\psi(t)=t/2$ and $\alpha(x,y) = 1/3$.
\end{proof}

The assertion that appears as Theorem~3.4 of~\cite{JR17b} can be corrected and
improved as discussed below. The assertion is
\begin{quote}
{\em Let $(X,d,\kappa)$ be a complete digital metric space and 
let $T: X \to X$ be a bijective and generalised $\alpha$-$\psi$-expansive mapping
that satisfies the following conditions:
\begin{enumerate}
\item $T^{-1}$ is $\alpha$-admissible;
\item there exists $x_0 \in X$ such that $\alpha(x_0,T^{-1}(x_0)) \ge 1$; and
\item $T$ is digitally continuous.
\end{enumerate}
Then $T$ has a fixed point.
} 
\end{quote}

\begin{remark}
We observe the following.
\begin{itemize}
\item The argument given for the assertion above in~\cite{JR17b} confuses
      topological continuity (the $\varepsilon$ - $\delta$ definition) and
      digital continuity (preservation of digital connectedness). 
      Therefore, item~3 above should be ``$T$ is topologically 
      continuous". 
\item The assumption of continuity can be omitted in common conditions, as shown below.
\item In the first line of the proof, ``$x_n=Tx_{n+1}$" is correct, but perhaps could
      be more clearly expressed as ``$x_{n+1}=T^{-1}x_n$ for $n>0$".
\item In statement~(4) of the proof, ``$d(Tx_{n-1},Tx_n)$" should be ``$d(x_{n-1},x_n)$".
\item In statement~(5) of the proof, the second ``$=$" should be ``$\ge$"
      and the two instances of ``$\le$" should each be ``$\ge$".
\item On the left side of the inequality in statement~(6) of the proof, the
      index ``$n+1$" should be ``$n-1$". Thus, the inequality should appear as
      \[ \psi(d(x_{n-1},x_n)) \ge \max\{d(x_n, x_{n-1}), d(x_n,x_{n+1})\}.
      \]
\end{itemize}
\end{remark}

The following is a corrected, somewhat modified, version of the
assertion above. Note we do not require $T$ to be continuous.

\begin{thm}
\label{genAlPsFixed}
Let $(X,d,\kappa)$ be a complete digital metric space and 
let $T: X \to X$ be a bijective and generalized $\alpha$-$\psi$-expansive mapping
that satisfies the following conditions:
\begin{enumerate}
\item $T^{-1}$ is $\alpha$-admissible;
\item there exists $x_0 \in X$ such that $\alpha(x_0,T^{-1}(x_0)) \ge 1$.
\end{enumerate}
Assume also that either $X$ is finite or $d$ is an $\ell_p$ metric.
Then $T$ has a fixed point.
\end{thm}

\begin{proof}
Our argument is based on its analog in~\cite{JR17b}.

We have hypothesized the existence of $x_0 \in X$ such that
$\alpha(x_0,T^{-1}(x_0)) \ge 1$. Define the sequence
$\{x_n\}_{n=0}^{\infty} \in X$ by $x_{n+1}=T^{-1}(x_n)$ for $n>0$. If
$x_{m+1}=x_m$ for some $m$, then $x_m$ is a fixed point of $T^{-1}$,
hence of $T$. Otherwise, $x_{n+1} \neq x_n$ for all $n$.

Since $T^{-1}$ is $\alpha$-admissible, we have that
$\alpha(x_0, x_1) = \alpha(x_0, T^{-1}(x_0)) \ge 1$ implies
$\alpha(x_1, x_2) = \alpha(T^{-1}(x_0), T^{-1}(x_1)) \ge 1$, and, by induction,
\[ \alpha(x_n, x_{n+1}) = \alpha(T^{-1}(x_{n-1}), T^{-1}(x_n)) \ge 1 \mbox{ for all } n.
\]
From Definition~\ref{genAlphaPsi}, we have
\[ \psi(d(x_{n-1},x_n)) = \psi(d(T(x_n),T(x_{n+1})) \ge \alpha(x_n,x_{n+1}) M(x_n,x_{n+1})
 \]
 \begin{equation}
 \label{psi-d-statement}
 \ge M(x_n,x_{n+1})
\end{equation}
where
\[ M(x_n,x_{n+1}) =  \max \{d(x_n,x_{n+1}), \frac{d(x_n,T(x_n)) + d(x_{n+1}, T(x_{n+1}))}{2}, \]
\[       ~~~~~~~~~~  \frac{d(x_n,T(x_{n+1})) + d(T(x_n),x_{n+1})}{2} \}
\]
\[ = \max\{d(x_n,x_{n+1}), \frac{d(x_n,x_{n-1}) + d(x_{n+1}, x_n)}{2}, \]
\[  ~~~~~~~~~~  \frac{d(x_n,x_n)) + d(x_{n-1},x_{n+1})}{2}\}
\]
\[ = \max\{d(x_n,x_{n+1}), \frac{d(x_n,x_{n-1}) + d(x_{n+1}, x_n)}{2},
   \frac{d(x_{n-1},x_{n+1})}{2}\}.
\]
Since by the triangle inequality, we have 
\[ \frac{d(x_{n-1},x_{n+1})}{2} \le \frac{d(x_n,x_{n-1}) + d(x_{n+1},x_n)}{2},
\]
it follows that
\[ M(x_n,x_{n+1}) \ge \max\{d(x_n,x_{n+1}), \frac{d(x_n,x_{n-1}) + d(x_{n+1}, x_n)}{2} \}
\]
\begin{equation}
\label{MLowerBound}
\ge \max\{d(x_n,x_{n+1}), d(x_n,x_{n-1}) \}.
\end{equation}
From inequalities~(\ref{psi-d-statement}) and~(\ref{MLowerBound}) we have
\[ \psi(d(x_{n-1},x_n)) \ge \max\{d(x_n,x_{n+1}), d(x_n,x_{n-1})\}.
\]
It follows from Proposition~\ref{psi1UpperBound} that
\begin{equation}
\label{maxDist}
\max\{d(x_n,x_{n+1}), d(x_n,x_{n-1}) \} = d(x_{n-1},x_n), \mbox{ i.e., }
     d(x_n,x_{n+1}) \le d(x_n,x_{n-1}).
\end{equation}
Applying statements~(\ref{psi-d-statement}),~(\ref{MLowerBound}),
and~(\ref{maxDist}), and using Proposition~\ref{psi1UpperBound}, gives
\begin{equation}
\label{psiAndNoPsi}
    d(x_n,x_{n+1}) \le \psi(d(x_{n-1},x_n)) \mbox{ for all } n \in \N.
\end{equation}
An easy induction based on~(\ref{psiAndNoPsi}) yields that
\[ d(x_n,x_{n+1}) \le \psi^n(d(x_0,x_1)) \mbox{ for all } n \in \N.
\]
By Proposition~\ref{psi1UpperBound}, it follows that 
$lim_{n \to \infty} d(x_n,x_{n+1}) = 0$. From Corollary~\ref{Han-Cauchy-cor},
we conclude that for large $n$ we have $T(x_{n+1}) = x_n = x_{n+1}$. Thus,
$x_{n+1}$ is a fixed point of $T$.
\end{proof}

Theorem~3.5 of~\cite{JR17b} is concerned with the existence of a fixed point for a
generalised $\alpha$-$\psi$-expansive mapping
satisfying conditions somewhat like those of Theorem~\ref{genAlPsFixed}.
The assertion is incorrectly stated (corrections noted below) as follows.
\begin{quote}
{\em Let $(X,d,\kappa)$ be a complete digital metric space. Suppose
$T: X \to X$ is a generalized $\alpha$-$\psi$ expansive mapping such that
\begin{enumerate}
\item $T^{-1}$ is $\alpha$-admissible;
\item there exists $x_0 \in X$ such that $\alpha(x_0, T^{-1}(x_0)) \ge 1$; and
\item if $\{x_n\}_{n=0}^{\infty} \subset X$ such that $\alpha(x_n,x_{n+1}) \ge 1$ 
for all $n$ and $\{x_n\}_{n=0}^{\infty}$ is digitally convergent to $x' \in X$, then 
$\alpha(T^{-1}(x_n),T^{-1}(x)) \ge 1$ for all $n$.
\end{enumerate}
Then $T$ has a fixed point.
} 
\end{quote}

\begin{remark}
Theorem~3.5 of~{\rm \cite{JR17b}} can be improved as follows.
\begin{itemize}
\item Since the existence of the function $T^{-1}$ is assumed, it should be stated
      that $T$ is assumed to be a bijection.
\item The term ``digitally convergent" is undefined. What the proof actually uses is
      metric convergence.
\item If we add the hypothesis that $X$ is finite or $d$ is an $\ell_p$ metric, then
      the assertion, as amended by these observations, follows from our 
      Theorem~\ref{genAlPsFixed}.
\end{itemize}
\end{remark}

\begin{remark}
Examples~3.6 and~3.7 of{\rm~\cite{JR17b}} claim $[0,\infty)$ as 
a digital metric space. Clearly, it is not.
\end{remark}

\begin{remark}
In the proof of Theorem~3.8 of~{\rm\cite{JR17b}}, the inequality
\[ d(u,T^n v) \le \psi(d(u,v)) \mbox{ for all } n=1,2,3, \ldots
\]
should be
\[ d(u,T^n v) \le \psi^n(d(u,v)) \mbox{ for all } n=1,2,3, \ldots
\]
\end{remark}

\section{Common fixed point assertions in~\cite{RaniJyo}}
The paper~\cite{RaniJyo} is concerned with common fixed points of pairs of
self-maps on digital metric spaces.

\begin{thm}
{\rm ~\cite{RaniJyo}} 
\label{RaniJyoCommonFP}
Let $T$ be a continuous mapping of a complete digital metric space $(X,d,\kappa)$
into itself. Then $T$ has a fixed point if and only if there exists $\alpha \in (0,1)$
and a function $S: X \to X$ that commutes with $T$ such that
\begin{equation}
\label{RJcondition}
S(X) \subset T(X) \mbox{ and } d(S(x), S(y)) \le \alpha d(T(x),T(y))
\mbox{ for all } x,y \in X.
\end{equation}
\end{thm}

\begin{remark}
Let $S$ and $T$ be as in Theorem~\ref{RaniJyoCommonFP}, where 
``continuous" is interpreted as $(c_u,c_u)$-continuous, $X$ is 
$c_u$-connected, and $0< \alpha < \frac{1}{u^{1/p}}$.
Then~$S$ must be a constant function.
\end{remark}

\begin{proof}
Let $x \adj_{c_u} x'$ in $X$. Since $T$ is $(c_u,c_u)$-continuous,
either $T(x)=T(x')$ or $T(x) \adj_{c_u} T(x')$, so $d(T(x),T(x')) \le u^{1/p}$. By the inequality~(\ref{RJcondition}),
$d(S(x),S(x')) \le  \alpha d(T(x),T(x')) < 1$. Since $d$ is an $\ell_p$ metric, $d(S(x),S(x'))=0$, so
$S(x) = S(x')$. It follows from the $c_u$-connectedness of $X$ that $S$ is constant.
\end{proof}

Below, we show that if we assume common conditions, the requirement that
$T$ be continuous in Theorem~\ref{RaniJyoCommonFP} is unnecessary. Our proof is 
similar to its analog in~\cite{RaniJyo}.

\begin{thm}
\label{myCommonFP}
Let $T$ be a mapping of a digital metric space $(X,d,\kappa)$ 
into itself, where $X$ is finite or $d$ is an $\ell_p$ metric.
Then $T$ has a fixed point if and only if there exists $\alpha \in (0,1)$
and a function $S: X \to X$ that commutes with $T$ such that
\begin{equation}
\label{myCopyRJcondition}
S(X) \subset T(X) \mbox{ and } d(S(x), S(y)) \le \alpha d(T(x),T(y))
\mbox{ for all } x,y \in X.
\end{equation}
\end{thm}

\begin{proof}
Suppose $T$ has a fixed point, say, $T(a)=a$ for some $a \in X$. Let
$S: X \to X$ be the constant function $S(x)=a$. Then clearly $S \circ T = T \circ S$
and the condition~(\ref{myCopyRJcondition}) is satisfied.

Suppose there exist a function $S: X \to X$ and $\alpha \in (0,1)$
such that $S \circ T = T \circ S$
and the condition~(\ref{myCopyRJcondition}) is satisfied. Let $x_0 \in X$. Since
$S(X) \subset T(X)$, there exists $x_1 \in X$ such that $T(x_1) = S(x_0)$, and,
inductively, for all $n \in \N$ there exists $x_n \in X$ such that
$T(x_n) = S(x_{n-1})$. It follows from (\ref{myCopyRJcondition}) that
\[ d(T(x_{n+1}),T(x_n)) = d(S(x_n),S(x_{n-1})) \le \alpha d(T(x_n), T(x_{n-1})).
\]
An easy induction yields that
\[ d(T(x_{n+1}),T(x_n)) \le \alpha^n d(T(x_1), T(x_0)).
\]
Since the right side of the latter inequality tends to 0 as $n \to \infty$, it follows
from Theorem~\ref{Han-Cauchy} that for sufficiently large $n$, $T(x_{n+1})=T(x_n)=t$ for
some $t \in X$. But $T(x_{n+1})=S(x_n)$, so for sufficiently large $n$, $S(x_n)=t$, and
since $S$ and $T$ commute,
\begin{equation}
\label{SandTCoincide}
T(t) = T(T(x_n)) = T(S(x_n)) = S(T(x_n)) = S(t).
\end{equation}
Thus, 
\[ T(T(t)) = T(S(t)) = S(T(t)),
\] so
\[ d(S(t), S(S(t)) \le \alpha d(T(t),T(S(t))) = \alpha d(S(t), S(S(t))), \mbox{ or}
\]
$0 \le (\alpha - 1) d(S(t), S(S(t)))$. Since the factor $\alpha - 1 < 0$, we must have
$d(S(t), S(S(t))) = 0$, so $S(t) = S(S(t))$. From~(\ref{SandTCoincide}) it follows
that $S(t) = S(S(t)) = S(T(t)) = T(S(t))$, so $S(t)$ is a common fixed point of 
$S$ and $T$.

Suppose $x$ and $y$ are common fixed points of $S$ and $T$. Then
\[ d(x,y) = d(S(x),S(y)) \le \alpha d(T(x), T(y)) = \alpha d(x,y).
\]
Since $0 < \alpha < 1$, we must have $d(x,y) = 0$, so $x=y$.
\end{proof}

Similarly, under common circumstances we can omit the 
assumption of continuity that is in the version of the 
following that appears in~\cite{RaniJyo} .

\begin{cor}
Let $T$ and $S$ be commuting mappings of a complete digital metric space $(X,d,\kappa)$
into itself, where $d$ is an $\ell_p$ metric.
Suppose that $S(X) \subset T(X)$. If there exists $\alpha \in (0,1)$
and a positive integer $k$  such that $d(S^k(x),S^k(y)) \le \alpha d(T(x),T(y))$ 
for all  $x,y \in X$, then $T$ and $S$ have a common fixed point.
\end{cor}

\begin{proof} Our proof is much like that of its analog in~\cite{RaniJyo}. 
Since $S$ and $T$ commute, $S^k$ and $T$ commute. Further, 
$S^k(X) \subset S(X) \subset T(X)$. Therefore, we can apply 
Theorem~\ref{myCommonFP} to the maps $S^k$ and $T$ to
conclude that there is a unique $a \in X$ that is a common fixed point of 
$S^k$ and $T$, i.e., $S^k(a) = T(a) = a$. Therefore,
\[ T(S(a)) = S(T(a)) = S(a) = S(S^k(a)) = S^k(S(a)),
\]
so $S(a)$ is a common fixed point of $T$ and $S^k$. But $a$ is the unique common fixed
point of $T$ and $S^k$, so $a=S(a)$.
\end{proof}

\begin{remark}
The assertion given as Corollary~3.2.5 of~{\rm \cite{RaniJyo}} has inequalities reversed.
\begin{itemize}
\item ``$K \ldots > 1$" should be $0 < K < 1$".
\item ``$d(S^n(x),S^n(y)) \ge Kd(x,y)$" should be ``$d(S^n(x),S^n(y)) \le Kd(x,y)$"
\end{itemize}
Thus, in order for the assertion to be a corollary of the preceding Theorem~3.2.4, 
the former should be stated as
\begin{quote}
Corollary. Let $n$ be a positive integer and let $0 < K < 1$. If 
$S$ is a self-map of a digital metric space $(X,d,\kappa)$ such that 
$d(S^n(x),S^n(y)) \le Kd(x,y)$ for all $x,y \in X$, then $S$ has a 
unique fixed point.
\end{quote}
\end{remark}

\begin{remark}
The assertion at Example~3.3.8 of~{\rm \cite{RaniJyo}} considers the maps
$S,T: (\Z,d,c_1) \to \Z$ given by $S(x) = 2-x^2$, $T(x)=x^2$ for 
all $x \in \Z$, where $d(x,y) = |x-y|$. It is claimed that
$d(S(x),S(y)) \le \frac{1}{2}d(T(x),T(y))$ for all $x,y \in \Z$, but 
this is clearly incorrect, since $d(S(x),S(y)) = d(T(x),T(y))$.
\end{remark}

\section{Common fixed point assertions in~\cite{SrideviKK}}
The paper~\cite{SrideviKK} is another that is concerned with common 
fixed points of pairs of self-maps on digital metric spaces. We have
the following.

\begin{definition}
{\rm ~\cite{SrideviKK}}
\label{weakCompatible} 
Let $(X,d,\kappa)$ be a digital metric space. Let $S,T: X \to X$.
Then $S$ and $T$ are {\em weakly compatible} if $S(x)=T(x)$ implies
$S(T(x)) = T(S(x))$.
\end{definition}

I.e., $S$ and $T$ are weakly compatible if they commute at all
coincidence points. Note this definition does not require that
coincidence points exist.

\begin{exl} The maps
$S,T: [0,1]_{\Z} \to [0,1]_{\Z}$ given by
$S(x) = 0$, $T(x) = 1$, are weakly compatible, despite having no
coincidence points, as they vacuously satisfy the requirement of
commuting at all coincidence points.
\end{exl}

Theorem~3.6, and Corollaries~3.7 and~3.8 of~\cite{SrideviKK} are
concerned with self-maps $S$ and $T$ on a digital metric space
$(X,d,\kappa)$ for which $d(S(x),S(y)) < d(T(x), T(y))$ for all 
$x,y \in X$ such that $x \neq y$. But such results may be quite 
limited under common conditions, as in the following.

\begin{prop}
Let $(X,d,\kappa)$ be a digital metric space with $|X| > 1$.
Let $S,T: \Z \to \Z$ be such that $x \neq y$ implies
$d(S(x),S(y)) < d(T(x),T(y))$. Then $T$ is one-to-one. If, further,
$X$ is finite, then $T$ is a bijection and $S$ is neither one-to-one
nor onto.
\end{prop}

\begin{proof}
Since $x \neq y$ implies $0 \le d(S(x),S(y)) < d(T(x),T(y))$, we 
have that $x \neq y$ implies $T(x) \neq T(y)$. Therefore, 
$T$ is one-to-one.

Suppose $X$ is finite. Since $T$ is one-to-one, it follows that $T$ is a bijection. Further, there exist $x_0,y_0 \in X$ such that
$d(x_0,y_0) = diam X$. If $S$ were onto, there would exist $x',y' \in X$
such that $S(x')=x_0$ and $S(y')=y_0$. By hypothesis, we would then have
\[ d(x_0,y_0) = d(S(x'),S(y')) < d(T(x'),T(y')),
\]
contrary to our choice of $x_0,y_0$. Therefore, $S$ is not onto.
Since $X$ is finite, it follows that $S$ is not one-to-one.
\end{proof}

\begin{prop}
Let $(\Z,d,c_1)$ be a digital metric space, where $d(x,y) = |x-y|$.
Let $S,T: \Z \to \Z$ be such that $x \neq y$ implies
$d(S(x),S(y)) < d(T(x),T(y))$. If $T$ is $c_1$-continuous, then
$S$ is a constant function.
\end{prop}

\begin{proof}
Since $T$ is $c_1$-continuous, for $x \adj_{c_1} x'$, we have 
$T(x) \adj_{c_1} T(x')$ or $T(x) = T(x')$, so $d(T(x),T(x')) \in \{0,1\}$.
Since $d(S(x),S(y)) < d(T(x),T(y))$, we must have $S(x)=S(x')$.
Since $\Z$ is $c_1$-connected, it follows that $S$ is a constant function.
\end{proof}

\section{Contractive mappings in~\cite{SrideviKKb}}
The paper~\cite{SrideviKKb} discusses fixed point assertions for
contractive-type mappings on digital metric spaces.

\begin{definition}
{\rm ~\cite{SrideviKKb}}
\label{phiContraction}
Suppose $(X,d,\kappa)$ is a digital metric space, $T: X \to X$,
and $\phi \in \Phi$. If
\[ d(T(x),T(y)) \le \phi(d(x,y)) \mbox{ for all } x,y \in X,
\]
then $T$ is called a {\em digital $\phi$-contraction}.
\end{definition}

\begin{remark}
The function of Example~\ref{counter} is a digital $\phi$-contraction
for $\phi(t) = t/2$. This shows that a digital $\phi$-contraction need
not be digitally continuous.
\end{remark}

A limitation of such functions is given in the following.

\begin{prop}
Let $(X,d,\kappa)$ be a digital metric space and let $T: X \to X$ be
a digital $\phi$-contraction for some $\phi \in \Phi$. Suppose
$d$ is an $\ell_p$ metric and $\phi(t) < 1$ 
for all $x,y \in X$. Then $T$ is a constant function.
\end{prop}

\begin{proof}
Let $x, x' \in X$. Then 
\[ d(T(x),T(x')) \le \phi(d(x,x')) < 1. \]
Therefore $T(x)=T(x')$. It follows that
$T$ is a constant function.
\end{proof}

\begin{definition}
\label{phiContraction}
Let $(X, d, \kappa)$ be a digital metric space, $T: X \to X$, and
$\phi \in \Phi$. We say that
\begin{itemize}
\item $T$ is {\em $\phi$-contractive} {\rm \cite{SrideviKKb}} if
$\phi(d(T(x), T(y))) < \phi(d(x, y))$ for all $x, y \in X$, $x \neq y$.
\item $T$ is a {\em digital contraction map} {\rm \cite{EgeKaraca}} if
      for some $\alpha \in (0,1)$, $d(T(x), T(y)) \le \alpha d(x,y)$
      for all $x, y \in X$.
\end{itemize}
\end{definition}

The similarity of these definitions yields the following.

\begin{prop}
Let $(X, d, \kappa)$ be a digital metric space with $X$ finite,
and let $T: X \to X$. If $T$ is digital contraction map then
for some $\phi \in \Phi$, $T$ is $\phi$-contractive.
The converse is true when $X$ is finite.
\end{prop}

\begin{proof}
Both assertions are trivial if $X$ is a singleton. Therefore, in
the following, assume $X$ is not a singleton.

Let $T$ be a digital contraction map. Then for some $\alpha \in (0,1)$,
$d(T(x), T(y)) < \alpha d(x,y)$ for all $x, y \in X$. 
Therefore, $\alpha d(T(x), T(y)) < \alpha^2 d(x,y)$. Note the function
$\phi(t) = \alpha t$ is a member of $\Phi$. Then $x \neq y$ implies
\[ \phi(d(T(x), T(y))) = \alpha d(T(x),T(y)) \le \alpha^2 d(x,y)
   < \alpha d(x,y) = \phi(d(x,y)),
\]
so $T$ is $\phi$-contractive.

Suppose $X$ is finite and $T$ is $\phi$-contractive for some
$\phi \in \Phi$. Then $x \neq y$ implies
\[ \phi(d(T(x), T(y))) < \phi(d(x, y)). \]
Since $\phi$ is monotone increasing, $x \neq y$ implies
\[d(T(x), T(y)) < d(x, y).\]
Since $X$ is finite, we can have $\alpha \in (0,1)$ well defined by
\[ \alpha = \max \left\{\frac{d(T(x), T(y))}{d(x, y)} \, | \,
                  x,y \in X, x \neq y \right \}. \]
Since $X$ is finite, $x \neq y$ implies
\[ \frac{d(T(x), T(y))}{d(x, y)} \le \alpha, \mbox{ or } 
    d(T(x), T(y)) \le \alpha d(x, y).
\]
Since the latter inequality also holds when $x=y$, it follows that
$T$ is a digital contraction map.
\end{proof}

Limitations on digital contraction maps are discussed in~\cite{BxSt18}.
Fixed point results for such maps appear in~\cite{EgeKaraca,SrideviKKb}.

\begin{definition}
{\rm \cite{SrideviKKb}}
Let $(X, d, \kappa)$ be a digital metric space. The function $T: X \to X$
is an {\em $\alpha$-$\psi$-$\phi$-contractive type mapping} if
there exist three functions $\alpha: X \times X \to [0,\infty)$ and
$\psi, \phi \in \Phi$ such that
\[ \alpha(x, y) \psi(d(Tx, Ty)) \le \psi(d(x, y)) - \phi(d(x, y))
   \mbox{ for all } x, y \in X.
\]
\end{definition}

\begin{remark}
One sees easily that every map $T: X \to X$ is an $\alpha$-$\psi$-$\phi$-contractive 
type mapping, for $\alpha(x,y) = 0$ and $\psi(t) \ge \phi(t)$. Hence
such a function $T$ need not be digitally continuous.
\end{remark}

Example~3.11 of~{\rm \cite{SrideviKKb}} has the following minor errors.
It is stated that $X=[0,\infty)$ (later corrected to 
$X=\{0,1,2,\ldots\}$). The statement
``Then T has [{\em sic}] $\alpha$-admissible" should be
``Then T is not $\alpha$-admissible".

The assertion stated as Theorem~3.12 of~{\rm \cite{SrideviKKb}} is both 
mis-stated and incorrect. Its mis-statement is in 
``$\alpha: [0,\infty) \to [0,\infty)$", which
should be ``$\alpha: X^2 \to [0,\infty)$". After making this correction,
the assertion would be as follows.
\begin{quote}
{\em
Let $(X,d,\kappa)$ be a digital metric space, $T: X \to X$, and
$\alpha: X^2 \to [0,\infty)$. Suppose
\begin{enumerate}
\item $T$ is $\alpha$-admissible;
\item there exists $x_0 \in X$ such that $\alpha(x_0,T(x_0)) \ge 1$;
\item $T$ is digitally continuous; and
\item $\alpha(x,y) \psi(d(T(x),T(y))) \le \psi(M(x,y)) - \phi(M(x,y))$,
      where
      \[M(x,y) = \]
      \[  \max\{d(x,y), \, d(x,T(x)), \, d(y,T(y)), \,
                   [d(x, T(y)) + d(y, T(x))] / 2\}.
      \]
      for all $x,y \in X$.
\end{enumerate}
Then $T$ has a fixed point. Further, if $u$ and $v$ are fixed points of
$T$ such that $\alpha(u,v) \ge 1$, then $u=v$.
} 
\end{quote}

This assertion is incorrect without additional hypotheses. For example,
if we take $X = [0,1]_{\Z}$, $T(x) = 1-x$, $\alpha(x,y)=1$, $\psi(x)=x$,
$\phi(x)=-1$, then all the hypotheses above are satisfied, but $T$ has
no fixed points.

If $X$ is finite or $d$ is an $\ell_p$ metric, then
the argument given as a proof for this assertion in~\cite{SrideviKKb}
does not require the continuity assumption, but does require that 
$\psi$ and $\phi$ have properties of $\Phi$. Thus, a correct, somewhat
modified version is as follows.

\begin{thm}
Let $(X,d,\kappa)$ be a digital metric space where $X$ is finite or
$d$ is an $\ell_p$ metric. Let $T: X \to X$, and
$\alpha: X^2 \to [0,\infty)$. Suppose
\begin{enumerate}
\item $T$ is $\alpha$-admissible;
\item there exists $x_0 \in X$ such that $\alpha(x_0,T(x_0)) \ge 1$; and
\item There exist $\psi, \phi \in \Phi$ such that
      \[ \alpha(x,y) \psi(d(T(x),T(y))) \le \psi(M(x,y)) - \phi(M(x,y)),\]
      where for all $x,y \in X$, $M(x,y) = $
      \[  \max\{d(x,y), \, d(x,T(x)), \, d(y,T(y)), \,
                   [d(x, T(y)) + d(y, T(x))] / 2\}.
      \]
\end{enumerate}
Then $T$ has a fixed point. Further, if $u$ and $v$ are fixed points of
$T$ such that $\alpha(u,v) \ge 1$, then $u=v$.
\end{thm}

\begin{proof}
Our argument is similar to its analog in~\cite{SrideviKKb}.

Let $x_0 \in X$ be such that $\alpha(x_0,T(x_0)) \ge 1$.
Let $x_n$ be defined inductively by $x_{n+1} = T(x_n)$ for $n \ge 0$.
Therefore, $\alpha(x_0,x_1) = \alpha(x_0,T(x_0)) \ge 1$, so
since $T$ is $\alpha$-admissible, we have
\[ \alpha(x_1,x_2) = \alpha(T(x_0),T(x_1)) \ge 1,
\]
and, inductively,
\[
\alpha(x_n,x_{n+1}) = \alpha(T(x_{n-1}),T(x_n)) \ge 1
\mbox{ for all } n.
\]
Then
\[ \psi(d(x_n,x_{n+1})) = \psi(d(T(x_{n-1}), T(x_n))) \le
  \alpha(x_{n-1},x_n) \psi(d(T(x_{n-1}), T(x_n)) 
\]
\[ \le \psi(M(x_{n-1},x_n)) - \phi(M(x_{n-1},x_n)) < \psi(M(x_{n-1},x_n)).
\]
Since $\psi$ is increasing,
\[   d(x_n,x_{n+1}) < M(x_{n-1},x_n) =
\]
\[ \max\{d(x_{n-1},x_n), d(x_{n-1},T(x_{n-1})),
   d(x_n,T(x_n)), \]
\[[d(x_{n-1},T(x_n)) + d(x_n,T(x_{n-1}))]/2\} =
\]
\[ \max\{d(x_{n-1},x_n), d(x_{n-1}, x_n), d(x_n, x_{n+1}),
   [d(x_{n-1},x_{n+1}) + d(x_n,x_n)]/2 \}=
\]
\[  \max\{d(x_{n-1},x_n), d(x_n, x_{n+1}), [d(x_{n-1},x_{n+1}) + 0]/2\} =
\]
\begin{equation}
\label{continueSec7-a}
 \max\{d(x_{n-1},x_n), d(x_n, x_{n+1}), d(x_{n-1},x_{n+1})/2\}
\end{equation}
By the triangle inequality,
\[ d(x_{n-1},x_{n+1})/2 \le [d(x_{n-1},x_n) + d(x_n, x_{n+1})]/2 
\]
\begin{equation}
\label{continueSec7-b}
     \le \max\{d(x_{n-1},x_n), d(x_n, x_{n+1})\}
\end{equation}
Combining~(\ref{continueSec7-a}) and~(\ref{continueSec7-b}),
$d(x_n,x_{n+1}) < \max\{d(x_{n-1},x_n), d(x_n, x_{n+1})\}$, so
\[ d(x_n,x_{n+1}) < d(x_{n-1},x_n). \]
Since $X$ is finite or $d$ is an $\ell_p$ metric, it follows that
for sufficiently large $n$, $d(x_n,x_{n+1}) = 0$, or
$x_n = x_{n+1} = T(x_n)$, so $x_n$ is a fixed point of $T$.

Suppose $u$ and $v$ are fixed points of $T$ with $\alpha(u,v) \ge 1$.
We have
\[ \psi(d(u,v)) = \psi(d(T(u),T(v)) \le \alpha(u,v) \psi(d(T(u), T(v))) \le \]
\[ \psi(M(u,v)) - \phi(M(u,v)) = \psi(d(u,v)) - \phi(d(u,v)),
\]
or $0 \le - \phi(d(u,v))$, which implies $d(u,v)=0$, or $u=v$.
\end{proof}

\section{Contractive maps in~\cite{Dol-Nal}}
In~\cite{BxSt18}, we discussed the paper~\cite{Dol-Nal}, including
mention of the fact that the author of the current work was identified
as a reviewer of~\cite{Dol-Nal} and that the latter work was published
without correction of several flaws mentioned in the review. Here, we
present additional discussion of~\cite{Dol-Nal}.

\begin{definition}
{\rm ~\cite{Dol-Nal}}
\label{WUSDCdef}
Let $(X, d,\kappa)$ be a digital metric space. A self map 
$T: X \to X$ is a {\em weakly uniformly strict digital contraction}
if given $\varepsilon > 0$, there exists $\delta > 0$ such that
$\varepsilon \le d(x, y) < \varepsilon + \delta$ implies 
$d(T(x), T(y)) < \varepsilon$ for all $x, y \in X$.
\end{definition}

\begin{lem}
\label{Dol-Nal-lem}
Let $(X, d,\kappa)$ be a digital metric space. Let 
$T: X \to X$ be a weakly uniformly strict digital contraction. Then
for all $x,y \in X$ such that $x \neq y$, $d(T(x),T(y)) < d(x,y)$.
\end{lem}

\begin{proof}
This is shown in the first paragraph of the argument given as a proof
of Theorem~3.3 in~\cite{Dol-Nal}.
\end{proof}

\begin{cor}
Let $(X, d,\kappa)$ be a digital metric space, such that $X$ is finite.
Let $T: X \to X$ be a weakly uniformly strict digital contraction. Then
$T$ is a digital contraction map.
\end{cor}

\begin{proof}
Without loss of generality, $T$ is not constant.
Since $X$ is finite,
\[ \alpha = \max \left \{ \frac{d(T(x),T(y))}{d(x,y)} ~| ~
       x,y \in X, x \neq y \right \}
\]
is well defined. Since $T$ is not constant, $\alpha > 0$; and,
since $X$ is finite, by Lemma~\ref{Dol-Nal-lem}, $\alpha < 1$.
Then for $x,y \in X$ such that $x \neq y$,
\[ d(T(x),T(y)) = \frac{d(T(x),T(y))}{d(x,y)}d(x,y) \le \alpha d(x,y).
\]
By Definition~\ref{phiContraction}, $T$ is a digital contraction map.
\end{proof}

\begin{prop}
Let $(X, d,\kappa)$ be a digital metric space such that 
$1 < |X| < \infty$.
Let $T: X \to X$ be a weakly uniformly strict digital contraction.
Then $T$ is neither one-to-one nor onto.
\end{prop}

\begin{proof}
Since $1 < |X| < \infty$, 
$m = \min\{d(x,y) \, | \, x,y \in X, x \neq y\}$ and
$M = \max\{d(x,y) \, | \, x,y \in X, x \neq y\}$
are well defined positive numbers, and there exist $x_0,y_0 \in X$
such that $x_0 \neq y_0$ and $d(x_0,y_0) = m$, and $x_1,y_1 \in X$
such that $x_1 \neq y_1$ and $d(x_1,y_1) = M$.

By Lemma~\ref{Dol-Nal-lem}, $d(T(x_0),T(y_0)) < d(x_0,y_0)=m$. By definition 
of~$m$, this implies $T(x_0)=T(y_0)$, so $T$ is not one-to-one.

If $T$ is onto, then there exist
$u,v \in X$ such that $T(u)=x_1$ and $T(v)=y_1$. Thus, by
Lemma~\ref{Dol-Nal-lem} we conclude that 
\[ M = d(x_1,y_1) = d(T(u), T(v)) < d(u,v),
\]
which contradicts our choice of~$M$. Therefore, $T$ is not onto.
\end{proof}

A limitation on weakly uniformly strict digital contractions is shown in the following.

\begin{prop}
\label{weaklyUnif}
Let $(X, d,c_u)$ be a digital metric space, where $d$ is an $\ell_p$ 
metric. Let $T: X \to X$ be a self map such that for some $\alpha > 0$,
$1 \le d(x,y) < u^{1/p}+\alpha$  
implies $d(T(x),T(y)) < 1$. If $X$ is $c_u$-connected, then $T$ is
constant.
\end{prop}

\begin{proof}
Let $x \adj_{c_u} y$ in $X$. Then for every $\alpha > 0$,
$1 \le d(x,y) < u^{1/p}+\alpha$, so $d(T(x),T(y)) < 1$. Therefore,
$T(x)=T(y)$. Since $X$ is $c_u$-connected, it follows that $T$ is
constant.
\end{proof}

As an immediate consequence, we have the following.

\begin{cor}
Let $(X, d,c_1)$ be a digital metric space, where $d$ is an $\ell_p$ 
metric. Let $T: X \to X$ be a weakly uniformly strict digital contraction.
If $X$ is $c_1$-connected, then $T$ is constant.
\end{cor}

\begin{proof}
Since $T$ is a weakly uniformly strict digital contraction,
for some $\alpha > 0$, $1 \le d(x,y) < 1+\alpha$  
implies $d(T(x),T(y)) < 1$. Since $d$ is an $\ell_p$ metric,
$c_1$-adjacent $x,y \in X$ satisfy $d(x,y)=1$, hence $d(T(x),T(y)) < 1$,
hence $T(x)=T(y)$. Since $X$ is $c_1$-connected, it follows from
Theorem~\ref{weaklyUnif} that $T$ is constant.
\end{proof}

The arguments given as proof for Theorems~3.1 and~3.3 in~\cite{Dol-Nal} 
are marred by confusion of digital and metric continuity. Below, we make 
minor modifications of the stated assumptions - in both theorems, we 
replace the assumption of a complete metric space
with the assumption of an $\ell_p$ metric - and give corrected proofs
that are much shorter than the arguments given in~\cite{Dol-Nal}, in
part because discussion of continuity is unnecessary.

The following is Theorem~3.1 of~\cite{Dol-Nal}, modified as discussed 
above. 

\begin{thm}
\label{DolNal3.1}
Let $(X, d, \kappa)$ be a digital metric space, where $d$ is an
$\ell_p$ metric, and suppose that
$T: X \to X$ satisfies $d(T(x), T(y)) \le \psi(d(x, y))$ for all 
$x,y \in X$, where $\psi: [0,\infty) \to [0,\infty)$ is monotone
nondecreasing and satisfies $lim_{n \to \infty} \psi^n(t) = 0$ for 
all $t > 0$. Then $T$ has a unique fixed point.
\end{thm}

\begin{proof}
Let $x_0 \in X$. Inductively, let $x_{n+1}=T(x_n)$ for $n \ge 0$. Then
\[ d(x_n,x_{n+1}) = d(T(x_{n-1}), T(x_n)) \le \psi(d(x_{n-1}, x_n))
\]
and a simple induction argument leads to the conclusion that 
for $n \ge 0$,
\[ d(x_n,x_{n+1}) \le \psi^n(d(x_0,x_1)) \to_{n \to \infty} 0.
\]
By Corollary~\ref{Han-Cauchy-cor}, for sufficiently large $n$
we have $x_n=x_{n+1}=T(x_n)$, so $x_n$ is a fixed point.

Suppose $x$ and $x'$ are fixed points of $T$. Then
\[ d(x,x') = d(T(x),T(x')) \le \psi(d(x,x')).
\]
where equality occurs only for $d(x,x')=0$, i.e., $x=x'$. Thus $T$
has a unique fixed point.
\end{proof}

\begin{remark}
If, in Theorem~\ref{DolNal3.1}, $d$ is an $\ell_p$
metric and $X$ is $c_1$-connected, then $T$ is a constant function.
\end{remark}

\begin{proof}
Given $x \adj_{c_1} x'$ in $X$, we have
\[ d(T(x), T(x')) \le \psi(d(x, x')) <~ \mbox{ (by Proposition~\ref{psi1UpperBound}) } ~d(x,x') = 1,
\]
so $T(x)=T(x')$. Since $X$ is $c_1$-connected, it follows that
$T$ is constant.
\end{proof}

The following is Theorem~3.3 of~\cite{Dol-Nal}, modified as discussed 
above.

\begin{thm}
Let $(X, d, \kappa)$ be a complete digital metric space, where
$d$ is an $\ell_p$ metric, and let
$T: X \to X$ be a weakly uniformly strict digital contraction mapping. 
Then $T$ has a unique fixed point $z$. Moreover, for any $x \in X$,
$lim_{n \to \infty}T^n(x)=z$.
\end{thm}

\begin{proof}
By Lemma~\ref{Dol-Nal-lem},
\[
\mbox{for all } x,y \in X, ~d(T(x),T(y)) < d(x,y).
\]
Let $x_0 \in X$ and, for $n \ge 0$, inductively define
$x_{n+1}= T(x_n)$. We may assume $x_1 \neq x_0$, since, otherwise,
$x_0$ is a fixed point of $T$. We have, for $n>0$,
$ d(x_n,x_{n+1}) = d(T(x_{n-1}), T(x_n)) < d(x_{n-1},x_n)$. Since
$d$ is an $\ell_p$ metric, for sufficiently large $n$,
$x_n = x_{n+1} = T(x_n)$; thus, $x_n$ is a fixed point of $T$.

Suppose $x$ and $y$ are fixed points of $T$. If $x \neq y$ then,
by Lemma~\ref{Dol-Nal-lem},
\[ d(x,y) = d(T(x),T(y)) < d(x,y),
\]
which is impossible. Therefore $x=y$.
\end{proof}

\section{Concluding remarks}
We have corrected or noted limitations of published assertions
concerning fixed points for self-maps on digital metric spaces.
This continues the work of~\cite{BxSt18}.
In many cases, flaws we have noted were so obvious that
reviewers and/or editors of these papers share responsibility
with the authors.

We have also offered improvements to some of the assertions discussed.

\end{document}